\documentclass{article}
\usepackage[utf8]{inputenc}
\usepackage{authblk}
\usepackage{graphicx}
\usepackage{amssymb}
\usepackage{amsfonts,amsthm,mathtools,verbatim,xcolor,verbatim,mathdots, mathrsfs}
\usepackage{stmaryrd}
\usepackage{tikz}
\usetikzlibrary{arrows.meta}
\usepackage{todonotes}

\usepackage[shortlabels]{enumitem}

\newtheorem{theorem}{Theorem}[section]

\newtheorem{lemma}[theorem]{Lemma}
\newtheorem{corollary}[theorem]{Corollary}

\theoremstyle{definition}

\newtheorem{definition}[theorem]{Definition}

\newtheorem{remark}[theorem]{Remark}

\newtheorem{question}[theorem]{Question}

\theoremstyle{remark}

\newtheorem{claim}{Claim}[theorem]
\newtheorem*{proofofclaim}{Proof of Claim}

\newcommand{\rca}{\mathrm{RCA}_0}
\newcommand{\rcas}{\mathrm{RCA}_0^*}
\newcommand{\wkl}{\mathrm{WKL}_0}
\newcommand{\wkls}{\mathrm{WKL}_0^*}

\newcommand{\cac}{\mathrm{CAC}}

\newcommand{\ads}{\mathrm{ADS}}

\newcommand{\EA}{{\ind\Delta_0}+{\exp}}

\newcommand{\srttt}{\mathrm{SRT}^2_2}

\newcommand{\rttt}{\mathrm{RT}^2_2}

\newcommand{\crttt}{\mathrm{CRT}^2_2}

\newcommand{\coh}{\mathrm{COH}}

\DeclareMathOperator{\Cond}{Cond}
\DeclareMathOperator{\Name}{Name}

\newcommand{\defd}{\mathclose\downarrow}
\newcommand{\condle}{\trianglelefteqslant}

\newcommand{\imp}{\rightarrow}

\newcommand{\cod}{\mathrm{Cod}}

\newcommand{\I}{\mathrm{I}}

\newcommand{\hyp}{\text{-}}

\newcommand{\fa}[1]{\forall{#1}\ }

\newcommand{\restdto}[1]{\mathclose{\restriction_{#1}}}

\newcommand{\restd}{\mathclose{\upharpoonright}}

\newcommand{\fain}[2]{\fa{{#1}\in{#2}}}

\newcommand{\IN}{\mathbb{N}}

\newcommand{\ee}{\mathrm{e}}

\renewcommand{\exp}{\mathrm{exp}}
\newcommand{\ind}{\mathrm{I}}
\newcommand{\bd}{\mathrm{B}}

\newcommand{\PA}{\mathrm{PA}}

\newcommand{\WKL}{\mathrm{WKL}}

\newcommand{\RCA}{\mathrm{RCA}}

\newcommand{\SRT}{\mathrm{SRT}}

\newcommand{\SC}{\mathrm{SC}}

\newcommand{\LPC}{\mathrm{LPC}}
\newcommand{\II}{\mathbb{I}}

\newcommand{\RT}{\mathrm{RT}}

\newcommand{\fincoh}{\mathrm{fin}\hyp\mathrm{COH}}

\newcommand{\tuple}[1]{\langle#1\rangle}
\newcommand{\Ack}[1]{\mathrm{Ack}(#1)}

\title{
The cohesive and stable Ramsey theorems \\ 
and proof size over a weak base theory
}

\author{Leszek Aleksander Kołodziejczyk\thanks{University of Warsaw, Institute of Mathematics, \texttt{lak@mimuw.edu.pl}}
~and Mengzhou Sun\thanks{University of Warsaw, Institute of Mathematics, \texttt{m.sun3@uw.edu.pl}}}

\date{October 28, 2026}

\begin{document}

\maketitle

\begin{abstract}

    We show that over the weak base theory $\rcas$,
    cohesive Ramsey's theorem for pairs $\crttt$
    implies exponential closure of the definable cut $\I^0_1$, which is the intersection of all $\Sigma^0_1$-definable cuts. Consequences include non-elementary proof speedup of $\rcas + \crttt$ over $\rcas$ for $\Pi_1$ sentences
    and the unprovability of $\crttt$ in $\rcas + \cac$.
    
    On the other hand, we show that $\rcas + \srttt$,
    where $\srttt$ is stable Ramsey's theorem for pairs,
    is polynomially simulated by $\rcas$ w.r.t.~proofs of $\forall \Pi^0_3$ sentences. Nevertheless, $\srttt$
    also implies a nontrivial property of $\I^0_1$, specifically closure under functions of quasipolynomial growth rate.
\end{abstract}

\section{Introduction}

The logical strength of combinatorial theorems
formalized in second-order arithmetic is a topic of considerable interest in reverse mathematics, and in the foundations of mathematics at large. A particularly interesting set of problems concerns the first-order, or more precisely $\Pi^1_1$, consequences 
of Ramsey's theorem for pairs and two colours, $\rttt$,
and of its natural weakenings and variants.
(See \cite{csy:inductive-ramsey, py:rt22, hpy:wo-rt22-conservation, hpy:rt22-pi4-conservation} for some research milestones in this area from the last decade.)

Starting with \cite{yokoyama:on-the-strength},
the strength of combinatorial statements has also been studied over the weak axiom system $\rcas$, which differs from the usual base theory for reverse mathematics, $\rca$, by disallowing uses of mathematical induction for non-computable (more precisely, $\Sigma^0_1$) properties. The study of $\rcas$ has led to new insights into reverse mathematics in the traditional setting of $\rca$ -- see \cite{belanger:coh, fkwy:wkl0star}. However, the strength of $\rttt$ and its cousins over $\rcas$ is also a natural and intriguing topic in its own right.

It was shown in \cite{fkk:weakcousins} that even though typical weakenings of $\rttt$ differ in terms of their first-order consequences over $\rcas$, in some sense  almost all of them have similar strength: they are $\forall\Pi^0_3$- but not $\Pi^1_1$-conservative over $\rcas$,
and thus also over the theory $\bd \Sigma^0_1 + \exp$, which captures the $\Pi^1_1$ consequences of $\rcas$.
(See \cite{sun:fin-coh} for a curious exception.)
That suggests the need for sharper tools to differentiate between the strength of various principles, and one such tool is provided by proof size analysis.

The study of proof size in axiomatic theories has a long tradition dating back to G\"odel \cite{goedel:laenge}. It was 
explicitly related to reverse mathematics by Avigad \cite{avigad:formalizing-forcing}, who introduced the 
method later called forcing interpretations in order to show
that the well-known $\Pi^1_1$-conservativity of $\wkl$
over $\rca$ can be witnessed by a polynomial proof transformation. More recently, proof size analysis has been
applied to the reverse mathematics of Ramsey theory \cite{kwy:ramsey-proof-size, kowalik:cac, hp:poly}.

In the specific area of reverse mathematics over $\rcas$,
it turns out that the effect on proof size is a good way of distinguishing between some Ramsey-theoretic statements 
with the same $\forall \Pi^0_3$ consequences. 
It was shown in \cite{kwy:ramsey-proof-size} that translating proofs of appropriately chosen $\forall \Pi^0_3$ (in fact, simpler) sentences from $\rcas + \rttt$ into $\rcas$ results in iterated exponential (``tower function'') blowup. A similar speedup phenomenon still holds for some weakenings of $\rttt$, like
the Erd\"os-Moser principle $\mathrm{EM}$. On the other hand,
Kowalik \cite{kowalik:cac} proved that the chain-antichain principle $\cac$ leads to no superpolynomial speedup over $\rcas$ for proofs of $\forall \Pi^0_3$ sentences. 

The proofs of both the speedup and the non-speedup results mentioned above make use of the definable cut $\I^0_1$, which is the intersection of all $\Sigma^0_1$-definable cuts and thus measures the available amount of $\Sigma^0_1$ induction. In general, $\I^0_1$ is only closed under multiplication. Adding $\rttt$ to $\rcas$ is enough to show closure of $\I^0_1$ under exponentiation, and that results in proof speedup, while $\cac$ does not imply any such additional closure properties, which plays a major role in the non-speedup argument. The effect of a Ramsey-theoretic principle on $\I^0_1$ depends in turn on the Ramsey numbers associated with its finite-combinatorial analogue. For example, Ramsey numbers for arbitrary 2-colourings of pairs are roughly exponential, and that is 
why $\rttt$ implies exponential closure of $\I^0_1$; but by Dilworth's theorem, Ramsey numbers for colourings corresponding to posets are only polynomial, and because of this $\cac$ does not prove superpolynomial closure of $\I^0_1$.  

A well-known pair of Ramsey-theoretic statements that has so far resisted attempts at such classification consists
of cohesive Ramsey's theorem for pairs, $\crttt$, and stable
Ramsey's theorem for pairs, $\srttt$. The principle $\srttt$ claims the existence of a homogeneous set for any colouring of pairs from $\IN$ that is \emph{stable}, meaning that for any given element, all but finitely many pairs involving that element have the same colour; $\crttt$ complements that by saying that for any colouring of pairs from $\IN$, there is an infinite set on which the colouring is stable. Intuitively, the reason why $\crttt$ and $\srttt$ have been difficult to analyze in terms of proof size is that it is not obvious what the finite analogue of a stable colouring should be. (On the syntactic level, this is reflected in the greater first-order quantifier complexity of $\crttt$ and $\srttt$ compared to many other variants of $\rttt$.)

In this paper, we resolve the conundrum concerning the effect of $\crttt$ and $\srttt$ on proof size over $\rcas$. We show that $\crttt$ behaves like $\rttt$ in that it  implies exponential closure of $\I^0_1$ and thus leads to iterated exponential proof speedup on rather simple sentences.
An interesting corollary of this result is that the implication $\cac \to \crttt$, known to hold over $\rca$, 
requires $\Sigma^0_1$ induction in the sense that it is not probable in $\rcas$. In contrast, $\srttt$ has no superpolynomial proof speedup over $\rcas$ for $\forall \Pi^0_3$ sentences. Unlike $\cac$, however, $\srttt$ does have nontrivial consequences for $\I^0_1$, as it implies the closure of that cut under functions of quasipolynomial growth rate.

The proofs of these facts once again involve finite combinatorics, though, as expected, this time the passage between infinite and finite is less direct than in the case of $\rttt$ and $\cac$. Roughly speaking, a colouring $f$ of pairs from $\{0,\ldots,m-1\}$ that is in a loose sense ``stable'' should have the property that for many $i < m$, the sequence $\tuple{f(i,i+1), f(i,i+2), \ldots, f(i,m-1)}$ contains long constant subsequences, and the related property that the same sequence rarely changes value. The former property is helpful in the analysis of $\crttt$, the latter -- in the analysis of $\srttt$, which is more subtle and requires proving bounds on Ramsey numbers for a parametrized family of classes of colourings.

The rest of this paper is structured as follows. We review the relevant background material in Section \ref{sec:prelim}. We prove the results on the strength of $\crttt$ in Section \ref{sec:crttt}. Passing to $\srttt$, we deal with the finite-combinatorial aspects and the effect on $\I^0_1$ in Section \ref{sec:bounds}, and in Section \ref{sec:forcing} we use that to prove the non-speedup result.

\section{Preliminaries}\label{sec:prelim}

We assume some basic familiarity with fragments of first- and second-order arithmetic, as presented in
\cite{book:HP} and in \cite{book:dzhafarov-mummert, book:simpson}, respectively.
Previous familiarity with the topic of proof size in axiomatic theories is not assumed, though it might make it easier to
follow the technical arguments in Section \ref{sec:forcing}
and in the proof of Corollary \ref{cor:crttt-speedup}. 
See~\cite{pudlak:handbook-lengths-of-proofs} for an excellent classical survey of the topic and the introductory parts of e.g.~\cite{kwy:ramsey-proof-size, kowalik:cac} for a brief but more up-to-date discussion.

Additionally, in the non-speedup proof in 
Section \ref{sec:forcing} we outsource tedious details familiar from earlier work to \cite{kowalik:cac}.
Nevertheless, even the reader who does not have \cite{kowalik:cac} at hand should be able to understand the gist of the proof and the technical aspects of all the novel parts.

The symbol $\IN$ denotes the set of natural numbers formalized in an arithmetic theory, while $\omega$ denotes the set of standard natural numbers. We use the convention that the letter $n$ always denotes a standard number.

The formula classes $\Sigma_n^0, \Pi_n^0$, and $\Delta_0^0$ are defined as usual by counting alternations of blocks of first-order quantifiers, allowing second-order variables but not second-order quantifiers.
Notation without the superscript ${}^0$, that is $\Sigma_n, \Pi_n, \Delta_0$, 
represents analogously defined classes of purely first-order formulas, in which second-order parameters are not allowed.
Given a set $A\subseteq\IN$, we write $\Sigma_n(A), \Pi_n(A)$ for the classes of $\Sigma_n^0$ and $\Pi_n^0$ formulas, respectively, in which $A$ is the only second-order parameter.
The class $\forall \Pi^0_n$ contains formulas that consist of a block of universal (first- and/or second-order) quantifiers
followed by a $\Pi^0_n$ formula.

It is well-known that the graph of exponentiation, viewed as a binary relation $2^x=y$, is $\Delta_0$-definable.
The axiom $\exp$ denotes the $\Pi_2$ statement $\forall x\, \exists y\; 2^x=y$ asserting the totality of exponentiation,
and the formula class $\Delta_0(\exp)$ is defined analogously to $\Delta_0$ but allowing the use of $2^x$ as a function symbol.
We write $2_a(x)$ to denote the $a$-th iteration of exponentiation, i.e., 
\[2_a(x) := 2^{2^{\scriptstyle 2^{\iddots^{\scriptstyle x}}}},\]
where the stack of exponents contains $a$ appearances of $2$.
The expression $2_a$ abbreviates $2_a(1)$. 
Functions like the logarithm $\log y$ and the iterated logarithm $\log^* y$ are treated as integer-valued, with
value equal to the greatest $x\in \IN$ such that $2^x\leq y$ resp.~$2_x\leq y$. As a result, some equalities in the proof of Corollary \ref{cor:srttt-qpoly-lb} hold only approximately. This causes no problems as what really matters for the proof is the \emph{in}equalities, and these hold as soon as the numbers involved are large enough. The function $\omega_1$
is defined by $\omega_1(x) = x^{\log x}$. A~function is said to have \emph{quasipolynomial} growth rate if it is bounded by some finite iteration of $\omega_1$, or equivalently by 
$x \mapsto x^{\log^\ell x}$ for some constant $\ell$.

The theory $\RCA^*_0$, originally introduced in
\cite{simpson-smith}, is obtained from $\RCA_0$ by weakening the $\Sigma^0_1$ induction axiom $\ind \Sigma^0_1$
to ${\ind\Delta^0_0} + {\exp}$. 
The theory $\WKL^*_0$ is obtained from $\WKL_0$ in an analogous way: in other words, $\WKL^*_0$ is $\RCA^*_0$ plus Weak K\"onig's Lemma $\WKL$.

Let $M$ be a model of $\ind\Delta_0+\exp$.
We say that a subset $I$ of $M$ is a \emph{cut} if it is closed under successor and downward closed; we denote this by $I\subseteq_\ee M$.
A cut $I\subseteq_\ee M$ is \emph{proper} if $I\neq M$.
If $(M, \mathcal{X})$ is a model of $\ind\Delta_0^0$, then no proper cut appears in the second-order universe $\mathcal{X}$, and the model fails to satisfy $\ind\Sigma_1^0$ if and only if a proper $\Sigma_1^0$-definable cut exists.

We write $\log I$ for the definable set $\{\log x\mid x\in I\}$, which need not be a cut unless the cut $I$ is closed under addition. The set $\log^*I$, which is a cut only if $I$ is closed under exponentiation, is defined similarly.

We can view an element $c$ of a model $M$ as representing the set $\Ack{c}$, where $x\in \Ack{c}$ is defined to mean that the $x$-th digit in the binary expansion of~$c$ is~$1$. We refer to sets of the form $\Ack{c}$ for $c \in M$ as \emph{$M$-finite sets}. 
Let $I$ be a proper cut of $M$.
Given a set $X\subseteq I$, we say that $X$ is \emph{coded} in $I$ if there is some $c\in M$ such that $X = \Ack c\cap I$.
We define $\cod(M/I) = \{\Ack c\cap I \mid c\in M\}$.
It is well-known~\cite{simpson-smith} that if $I$ is exponentially closed, then $(I, \cod(M/I))$ is a model of $\wkl^*$.

We say that a set $A$ is \emph{unbounded} in $\IN$
(or \emph{infinite})
if $\forall b \, \exists x \! > \! b \; x\in A$.
In a model $(M, \mathcal{X}) \vDash \rcas$,
the sets in $\mathcal{X}$ that are not unbounded are
exactly the $M$-finite sets. Over $\RCA_0^*$, an unbounded set $A\subseteq\IN$ need not be in bijective correspondence with $\IN$.
Instead, there is a unique cut $J$, called the \emph{cardinality} of $A$, such that there is an increasing enumeration $\{a_i\mid i\in J\}$ of the elements of $A$.
Moreover, the cut $J$ 
is $\Sigma_1^0$-definable, namely,
\[i\in J \text{ iff } \exists x \! \in \! A \; x=a_i.\]
Conversely, every $\Sigma_1^0$-definable cut can be obtained in this way for some unbounded set $A$.

The cut $\I_1^0$ is defined to consist of those $x\in\IN$ 
such that every unbounded $A\subseteq\IN$ contains a finite subset of cardinality $x$.
By the previous paragraph, $\I_1^0$ is exactly the intersection of all $\Sigma_1^0$-definable cuts.
In particular, it is provable in $\rcas$ that 
for any $x\in \IN$ and $v\in \I_1^0$, the number $2_v(x)$ exists; also, $\I_1^0=\IN$ if and only if $\RCA_0$ holds.

We use the letters $\sigma, \tau$ to denote finite sequences, which are represented by numbers in the usual way.
We use the symbol $\emptyset$ to stand for the empty sequence.
The notation $\sigma\sqsubseteq\tau$ means that $\sigma$ is an initial segment of $\tau$. For concatenation with a one-element sequence, we write $\sigma i$
instead of $\sigma^\smallfrown\langle i\rangle$.
We will use vertical lines, as in $|\sigma|$ or $|S|$, 
to stand for the length of a finite sequence or for
the cardinality of a finite set, depending on the context.

In contrast to the standard finite-combinatorial convention, we write $[m]$ for the set $\{0,1,\dots,m-1\}$ instead of $\{1,\dots,m\}$. For an arbitrary set $A$, the notation $[A]^2$ stands for the set of unordered pairs of distinct elements of $A$, represented as usual by (codes of) ordered pairs with the first element smaller than the second one.
We write $[m]^2$ instead of $[[m]]^2$.

If a 
set $A \subseteq \IN$ is enumerated in increasing order as $A=\{a_0<a_1<\ldots\}$ and
$f\colon [A]^2\to 2$, then
$f(a_i,-)$ stands for the 
binary sequence
\[\tuple{f(a_i,a_{i+1}),f(a_i,a_{i+2}),\ldots}.\]
This convention applies also when $A$ is finite,
in which case the enumeration is $\{a_0<a_1<\ldots<a_k\}$ for some $k$, and $f(a_i,-)$ is the finite sequence
$\tuple{f(a_i,a_{i+1}),\ldots,f(a_i,a_k)}$.
If $A\subseteq\IN$ is unbounded,
we say that a colouring $f\colon [A]^2\to 2$ is \emph{stable} if  for every $x\in A$, the sequence $f(x,-)$ is eventually constant.

Ramsey's Theorem for pairs and two colours, $\RT_2^2$, 
asserts that for any colouring $f\colon [\IN]^2\to 2$ there is an unbounded \emph{homogeneous} set $H$, i.e.,
\[\fain {x,y,z,w} H f(x,y)=f(z,w).\]
Cholak, Jockusch, and Slaman~\cite{cjs:ramsey} decomposed $\rttt$ into two weaker principles, Cohesive Ramsey's Theorem for pairs, $\crttt$, and Stable Ramsey's Theorem for pairs, $\SRT^2_2$:

\begin{itemize}[leftmargin = 20 mm]
\item[$\crttt$:~~]  For every $f \colon [\IN]^2\to 2$ there exists an unbounded set $C \subseteq \IN$ \\ such that $f\restdto {[C]^2}$ is stable.

\item[$\srttt$:~~] For every stable $f \colon [\IN]^2\to 2$ there exists an unbounded \\ homogeneous set $H \subseteq \IN$.
\end{itemize}

Two other Ramsey-theoretic statements that will be mentioned are the chain-antichain principle $\cac$, which says that any partial order on $\IN$ contains an infinite chain or an infinite antichain, and the ascending-descending sequence principle $\ads$, which says that for any linear order on $\IN$ there is an infinite set on which the order either agrees with the usual natural number ordering $\le$
or agrees with the inverse of $\le$. 

Provably in $\rcas$, each of the above Ramsey-theoretic principles is equivalent to its generalization
to instances defined on arbitrary unbounded subsets of $\IN$. 
Over $\rca$, we have the following sequence of implications:
\[ \rttt \to \cac \to \ads \to \crttt, \]
proved mostly in \cite{hirschfeldt-shore}. All of these except $\ads \to \crttt$ and $\cac \to \crttt$ are known to hold over $\rcas$ as well.

When considering questions related to proof size, formally speaking we should fix a specific proof system, but the exact choice is immaterial as there are polynomial-time translations between typical Hilbert-style, natural deduction, and sequent calculus systems with cut \cite{eder:complexities-calculi}. For definiteness, we assume that our proof system is like the one in \cite[Section 2.4]{book:enderton} except that propositional tautologies are derived from a finite number of axiom schemes rather than assumed as axioms. As in the system from \cite{book:enderton},
our official set of logical symbols consists only of 
of $\neg$, $\to$, and $\forall$, with the others treated as abbreviations.

A theory $S$ \emph{polynomially simulates} a theory $T$ with respect to sentences from the class $\Gamma$ if there is a polynomial-time procedure that takes a $T$-proof of a sentence
$\gamma \in \Gamma$ and outputs an $S$-proof of $\gamma$.
In contrast, $T$ has \emph{non-elementary speedup} over $S$ with respect to $\Gamma$ if there is no elementary recursive function $f$ such that if $\pi$ is a $T$-proof of a sentence $\gamma \in \Gamma$, then there exists an $S$-proof of $\gamma$ of size at most $f(|\pi|)$.

\section{The strength of Cohesive Ramsey's Theorem}\label{sec:crttt}

In this section, we study the behaviour of $\crttt$ over $\rcas$. We show that $\crttt$ 
shares some features of $\rttt$ that make it appear ``strong'' in this context, distinguishing it from other weakenings of $\rttt$ like $\cac$ or $\ads$. 

\begin{theorem}\label{thm:crttt->i01-expclosed}
The theory $\rcas + \crttt$ proves that $\I^0_1$ is closed under $\exp$.
\end{theorem}

\begin{proof}
We present the proof in a model-theoretic setting, 
working with a model of $\rcas$ rather than
in the theory, but we emphasize that this is a matter of convenience rather than necessity. The argument could be given a (perhaps somewhat unwieldy) purely syntactic form.

Let $(L, \cal X)$ be a model of $\rcas$ with $\I^0_1$ not closed under $\exp$. 
Then obviously $(L, \mathcal{X}) \not \vDash \rca$, since $\rca$ proves $\I^0_1 = \IN$. As a consequence (see \cite[Lemma 9]{ky:categorical}) there is a proper exponential $\Sigma^0_1$-cut $M$ in $(L, \cal X)$. 
As mentioned in the Preliminaries, the structure 
$(M, \cod(L/M))$ is a model of $\wkls$.
One can check that $\I^0_1(M,\cod(L/M)) = \I^0_1(L, \cal X)$; in the remainder of this argument, we will refer to
$\I^0_1(M,\cod(L/M))$ simply as $\I^0_1$. 
Moreover, by \cite[Corollary 3.10]{fkk:weakcousins}, $(M, \cod(L/M))$ satisfies $\crttt$ if and only if $(L, \cal X)$ does. So, to prove that 
$(L, \mathcal{X}) \not \vDash \crttt$ it suffices to show 
$(M, \cod(L/M)) \not \vDash \crttt$.
Given an $L$-finite set $S$, we will write $\widehat S$ for $S \cap M$, which is an element of $\cod (L/M)$.

Since $\I^0_1$ is not closed under $\exp$,
there is some $k \in M$ such that $k \in \I^0_1 < 2^{k/4}$. Fix some such $k$.
By the assumption that $\I^0_1 < 2^{k/4}$, we can also fix
an $L$-finite set $A$ of cardinality $2^{k/4}$ such that $\widehat A$ is cofinal in $M$. 

To prove that $(M, \cod(L/M)) \not \vDash \crttt$, we will use a probabilistic argument in $L$. More precisely, we will isolate a 
finite-combinatorial (from the point of view of $L$) property of colourings $f \colon [A]^2 \to 2$ 
that occurs with non-zero probability for random $f$
and has the infinite-combinatorial (from the point of view of $M$) consequence that $\widehat f$ is an instance of $\crttt$ with no solution in $(M,\cod(L/M))$. 

Consider a fixed $f \colon [A]^2 \to 2$ and a hypothetical set $C \subseteq A$ such that $\widehat C$ is a solution to $\crttt$
on instance $\widehat f$. We make the following observations:

\begin{enumerate}[(i)]

    \item As $\widehat C$ is unbounded in $M$ and $C\subseteq A$, the cardinality of $\widehat C$ is a $\Sigma^0_1$-cut $J_C$ such that $k \in J_C < 2^{k/4}$. Moreover, $J_C$ is not closed under addition, because otherwise $\log J_C$ would be a $\Sigma^0_1$-cut with $\log J_C < k$, contradicting the assumption that $k \in \I^0_1$.
    
    \item Thus, by possibly taking some elements out of $C \setminus \widehat C$, we may assume that there is some $j$
    satisfying 
    \[\I^0_1 < j < J_C < |C| = 2j \le 2^{k/4}.\]
    
    \item Let $c_1, \ldots, c_{2j}$ be the elements of $C$ enumerated in increasing order; observe that all of $c_1, \ldots, c_j$ belong to $\widehat C$. For each $i = 1,\ldots, j$, consider the $j$-element binary sequence
    \[s_i:= 
    \tuple{ f(c_i,c_{j+1}), \ldots, f(c_i,c_{2j}) } \]
    (that is, the sequence $f(c_i,-)$ restricted to the positions corresponding to $c_{j+1},\ldots, c_{2j}$). Note that $s_i$ is the characteristic
    function of the set
    \[S_i = \{y \in C \mid y > c_j \land f(c_i,y) = 1\}.\] 
    By the assumption that $\widehat{C}$ is assumed to be a solution to $\crttt$ for $\widehat{f}$, we know that for each $i$ either $\widehat{S_i}$ or $\widehat{C}\setminus \widehat{S_i}$ is cofinal in $M$. Hence, since $k+1 \in \I^0_1$, each $s_i$ must contain a run of at least $k+1$ consecutive 1's or of at least $k+1$ consecutive 0's.

\end{enumerate}

Let us now assume that we specify $f \colon [A]^2 \to 2$
by a random process in which $f(x,y)$ is determined by flipping a fair coin, independently for each $\{x,y\} \in [A]^2$.
As mentioned, in order to prove that $(M, \cod(L/M)) \not \vDash \crttt$, and thus complete the argument, we want to find an event with probability strictly less than 1 that is implied by the existence of a solution to $\crttt$ for $\widehat f$ in $\mathcal X$. 

Fix some $j \in [k,2^{k/4}]$ and a set $C$ of cardinality $2j$ as above. 
By observation (iii), if $\widehat C$ is a solution to $\crttt$ for 
$\widehat f$, then for each $i = 1, \ldots, j$ the following event $E_{j,C,i}$ occurs: the string $s_i$ contains a run of 1's or a run of 0's of length $k+1$.

We claim that for any fixed $i$,
the probability of $E_{j,C,i}$ is less than $\frac{j}{2^k}$. 
To see this, note that the probability that we get a run of length $k+1$ starting at any fixed position in $s_i$ is exactly $\frac{1}{2^k}$, and $s_i$ has length $j$, so there are fewer than $j$ possible starting positions for the run. Thus, the union bound gives $Pr(E_{j,C,i}) < \frac{j}{2^k}$, as claimed. 

For fixed $j, C$, the events $E_{j,C,i}$ with $i$ ranging in $\{1,\ldots,j\}$ are independent, so we get
\[\Pr\left(\bigcap_i E_{j,C,i}\right) < \left(\frac{j}{2^k}\right)^j.\]

Now, for $j \in [k,2^{k/4}]$ still fixed, consider the event $E_j$ := ``there is \emph{some} $C \subseteq A$ such that $|C| = 2j$ and $\bigcap_j E_{j,C,i}$ occurs''. The number of $2j$-element subsets of $A$ is
\[\binom{2^{k/4}}{2j} \le \left(2^{k/4}\right)^{2j} 
= 2^{kj/2},\]
so, by the union bound again, we get 
\[\Pr(E_j) < 2^{kj/2}\left(\frac{j}{2^k}\right)^j = \left(\frac{j2^{k/2}}{2^k}\right)^j \le \left(\frac{2^{3k/4}}{2^k}\right)^j = \left(\frac{1}{2^{k/4}}\right)^j \le \frac{1}{2^{k^2/4}}.\]
Here, the second inequality holds because 
$j \le 2^{k/4}$, and the last one holds because $j \ge k $.

Finally, using the union bound one more time, we get

\[\Pr\left(\bigcup_{j=k}^{2^{k/4}}E_j\right) < {2^{k/4}}\frac{1}{2^{{k^2}/{4}}} = {2^{-\frac{k^2-k}{4}}} \ll 1.\]

But if $f$ is such that a solution to $\crttt$ for $\widehat f$ exists, then $\bigcup_{j=k}^{2^{k/4}}E_j$ must hold for $f$. This concludes the proof.
\end{proof}

We now discuss a few consequences of Theorem
\ref{thm:crttt->i01-expclosed}. Firstly,
$\crttt$ gives significant proof speedup
over $\rcas$ with respect to proofs 
of $\forall \Pi^0_3$, and in fact much simpler, sentences. This answers a question asked in \cite{kowalik:cac}.
Secondly, we are able to answer two questions
left open in \cite{fkk:weakcousins}: 
whether the implication $\ads \to \crttt$ or at least
$\cac \to \crttt$ is provable in $\rcas$, 
and whether $\crttt$ is $\Pi_4$-conservative
over $\rcas$ (it was shown not to be
$\Pi_5$-conservative in \cite{fkk:weakcousins}).

\begin{corollary}\label{cor:crttt-speedup}
$\rcas + \crttt$ has non-elementary speedup over $\rcas$ with respect to proofs of $\Delta_0(\exp)$ sentences.
\end{corollary}

\begin{proof}
The proof is essentially the same as for $\rttt$ in \cite[Section 3]{kwy:ramsey-proof-size} -- in \cite{kwy:ramsey-proof-size}, a detailed  argument is given for proofs of $\Sigma_1$ sentences, but it remarked there that the speedup applies also to finite consistency statements that can be expressed in a $\Delta_0(\exp)$ way.

Let $\mathrm{Con}_n(T)$ be the statement
``there is no inconsistency proof in $T$ with at most $n$ symbols'', and for a definable cut $J$,
let $\mathrm{Con}^J(T)$ mean ``there is no inconsistency proof in $T$ that belongs to $J$''. Since $\rcas + \crttt$ proves that the cut $\I^0_1$ is exponentially closed, it also proves that 
$\log^*(\I^0_1)$ is a cut. But it is known
that $\mathrm{Con}^{\log^*(\I^0_1)}(\EA)$ holds provably in $\rcas$. The argument is essentially as follows: if $2_k \in \I^0_1$ then $2_{2_k}$ exists, which means that the exponential function, and therefore every Skolem function for a fixed finite axiomatization of $\EA$, can be iterated $2_k$ times. So, there is no cut-free
proof of inconsistency in $\EA$ with at most $2_k$ symbols, and by the known upper bounds for cut elimination,
there is no proof of inconsistency in $\EA$ with at most $k$ symbols; \emph{a fortiori}, there cannot be one with code below $k$.

Thus, $\rcas + \crttt$ proves the consistency of $\EA$ on a cut. As a consequence (cf.~\cite{pudlak:handbook-lengths-of-proofs} for this sort of argument) it has iterated exponential speedup
over $\EA$ on statements of the form $\mathrm{Con}_{2_n}(\EA)$, which can be expressed as 
$\Delta_0(\exp)$ sentences of size $O(n)$. 
However, by \cite[Theorem 4.9]{kwy:ramsey-proof-size}, 
$\rcas$ has at most polynomial speedup
over $\EA$ with respect to proofs of $\Pi_2$ statements,
so the iterated exponential speedup
of $\rcas + \crttt$ over $\EA$ for the 
$\mathrm{Con}_{2_n}(\EA)$ statements
implies analogous speedup over $\rcas$ as well.
\end{proof}

\begin{corollary}
$\rcas + \cac$ does not prove $\crttt$.
\end{corollary}
\begin{proof}
By \cite[Theorem 3.16]{fkk:weakcousins}, 
$\rcas + \cac$ does not imply
exponential closure of $\I^0_1$. (In fact, by \cite{kowalik:cac}, it does not have superpolynomial
speedup over $\rcas$.)
\end{proof}

\begin{corollary}\label{cor:crttt-not-conservative}
$\rcas + \crttt$ is not $\Pi_4$-conservative over $\rcas$.
\end{corollary}
\begin{proof}
It was shown in \cite[Lemma 5.4]{fkwy:wkl0star}
that if $(M,\mathcal{X}) \vDash \rcas$ and $A \in \mathcal{X}$ is such that $\neg \ind \Sigma_1(A)$ holds,
then every $\Sigma^0_1(M,\mathcal{X})$-definable proper cut in $M$
is in fact $\Sigma_1(A)$-definable. As a result, 
it is provable in $\rcas$ that if $\neg \mathrm{I}\Sigma_1$ holds, then
$\I^0_1$ is equal to the (lightface) $\Pi_3$-definable cut of those $k$ such that every unbounded $\Delta_1$-set has a $k$-element subset. 

The statement ``if $\neg \mathrm{I}\Sigma_1$ holds, then for every $k$,  
if every unbounded $\Delta_1$-definable set has a $k$-element subset, then every unbounded $\Delta_1$-definable set has 
a $2^k$-element subset'' is $\Pi_4$. By Theorem \ref{thm:crttt->i01-expclosed} and the discussion above, this statement is provable in $\rcas + \crttt$. On the other hand,
it is unprovable in $\rcas$, essentially by 
the argument from the proof of Corollary \ref{cor:crttt-speedup}.
\end{proof}

We conclude this section with a remark and a question. 

\begin{remark}\label{rem:fin-coh}
This remark is largely inspired by some comments and questions of Keita Yokoyama (private communication). It was shown in \cite{sun:fin-coh} that $\fincoh$, the restriction of the well-known cohesive set principle $\coh$ (which is a strengthening of $\crttt$) to instances that are finite families of sets, implies $\rca$ over $\rcas$. 
Looking at conservativity and proof size lets us give a finer analysis of the strength of $\fincoh$.

As suggested by Yokoyama, 
a simple modification of the proof of Theorem \ref{thm:crttt->i01-expclosed} shows that
if $\ell > \I^0_1$ and each $\coh$ instance of the form $(R_i)_{i < \ell}$ has a solution, then $\I^0_1$ is exponentially closed. 
However, the argument from \cite{sun:fin-coh} directly shows that if $\ell > \I^0_1$ and each $\coh$ instance of the form $(R_i)_{i < 2^\ell}$ has a solution, then $\rca$ holds. Thus, the existence of solutions to all $\coh$ instances of the form $(R_i)_{i < \ell}$ for even a single $\ell > \I^0_1$ implies $\rca$.

On the other hand, let $\I^0_1\text{-}\coh$ be the restriction of $\fincoh$ to instances of the form $(R_i)_{i < k}$ for $k \in \I^0_1$,
and let $(\log \I^0_1)\text{-}\coh$ be the further restriction requiring $2^k \in \I^0_1$. Then a standard argument shows that $\I^0_1\text{-}\coh$ is $\forall\Pi^0_3$-conservative over $\rcas$ and hence does not prove $\rca$. Moreover, by adapting the well-known proof of $\rca + \ads \vdash \coh$ from \cite{hirschfeldt-shore}, it is possible to prove 
$\rcas + \ads \vdash (\log \I^0_1)\text{-}\coh$.
Thus, by \cite{kowalik:cac}, $(\log \I^0_1)\text{-}\coh$
does not lead to any nontrivial proof speedup over $\rcas$ for $\forall\Pi^0_3$ statements. 

This leaves the problem of whether the $\forall\Pi^0_3$-conservativity of full
$\I^0_1\text{-}\coh$ over $\rcas$ can also be witnessed by a polynomial-size proof transformation. By a more refined version of the argument proving Theorem \ref{thm:crttt->i01-expclosed}, which we will not discuss here, we can at least show that $\I^0_1\text{-}\coh$ implies exponential closure of $\I^0_1$ over $\rcas + \ads$. 
So, like in the proof of Corollary \ref{cor:crttt-speedup}, $\rcas + \ads + \I^0_1\text{-}\coh$ has non-elementary speedup over $\rcas + \ads$ (thus also over $\rcas$) w.r.t.~proofs of $\Delta_0(\exp)$ sentences. 
\end{remark}

\begin{question}
Does $\ads + \text{``}\I^0_1 \text{ is closed under } \exp\text{''}$ imply $\crttt$ over $\rcas$? The same question can be asked for $\cac$ instead of $\ads$.
\end{question}

\section{Stable Ramsey's Theorem: upper and lower bounds}\label{sec:bounds}

We would now like to investigate the effect of $\srttt$ 
on proof size over $\rcas$. Before considering that topic directly, however, we need to study finite colourings that turn out to be related to $\srttt$, and to prove bounds on Ramsey numbers for such colourings. This also leads to a characterization of the closure properties of
$\I^0_1$ implied by $\srttt$.

We start by isolating the relevant class of finite colourings -- in fact, not a single class but a parametrized family of classes -- and proving a lower bound on the associated Ramsey numbers.

\begin{definition}
    Let $A \subseteq \IN$ be finite, $A = \{a_0 < a_1 <\ldots < a_\ell\}$, let $f\colon [A]^2\rightarrow 2$,
    and let $k \in \IN$. We say that $f$ is \emph{at most $k$-unstable} if for every $i$, the sequence $f(a_i,-)$ contains at most $k$ alternating blocks of successive 0's or successive 1's: that is, if there are at most $k-1$ numbers $j$, with $i<j<\ell$, such that $f(a_i,a_j) \neq f(a_i,a_{j+1})$.
\end{definition}

\begin{lemma}\label{lem:SRT_lb}
    $(\EA)$ 
    For any $s\geq 4$ and $t\in\IN$, let 
    \[k=st, \ell=(2\log s)^t, m=s^t.\]
    Then, there exists a colouring $f\colon [m]^2\to 2$ such that 
    \begin{enumerate}[(a)]
        \begin{samepage}
        \item $f$ is at most $k$-unstable,
        \item there is no homogeneous subset for $f$ of size $\ell$.
        \end{samepage}
    \end{enumerate}
\end{lemma}  
\begin{proof}
    Using the fact that $s \ge 4$ and standard lower bounds on Ramsey numbers (which are provable in $\EA$),
    fix some colouring $g$ on $[s]^2$ for which there is no homogeneous set of size $2\log s$.
    
    We define the colouring $f$ recursively in $t+1$ steps
    numbered $0,\ldots,t$.
    In each step $i$, we will divide $[m]$ into $s^i$ many intervals $I_\sigma$, each of length $s^{t-i}$, where the index $\sigma$ ranges over all $s$-ary sequences of length $i$.
    We will also define the colouring $f$ for pairs of elements that are in distinct intervals at this step but not at earlier steps.
    The initial stages of the construction are illustrated 
    in Figure \ref{fig:lb}.
    
    \begin{figure}
      \centering
      \begin{tikzpicture}[x=1.05cm,y=1.05cm,line cap=round,line join=round,>=Latex, scale=0.9, transform shape]

\def\REDOPACITY{0.6}

  \draw[black, thick] (0,0) -- (8,0);
  \foreach \x in {0,2,4,6,8}{
    \draw[black, thin] (\x,-0.14) -- (\x,0.14);
  }
\node at (1,-0.24) {$I_0$};
\node at (3,-0.24) {$I_1$};
\node at (5,-0.24) {$I_2$};
\node at (7,-0.24) {$I_3$};
\node[font=\large] at (7,2) {$I_\emptyset$};
  \foreach \a/\b in {0.42/1.58,2.42/3.58,4.42/5.58,6.42/7.58}{
    \draw[red][line width=1.2, draw opacity=0.6] (\a-0.1, 0.24) -- (\b+0.1, 0.24);
    \draw[blue][line width=1.2, draw opacity=1] (\a-0.1,-0.45) -- (\b+0.1,-0.45);
  }

  \foreach \x in {4.5,5.0,5.5}{
    \draw[black, thin] (\x,-0.07) -- (\x,0.07);
  }

  \draw[red, line width=0.8, draw opacity=\REDOPACITY] (1.00,0.24) to[out=80,in=100] (3.00,0.24);
  \draw[red, line width=0.8, draw opacity=\REDOPACITY] (3.00,0.24) to[out=80,in=100] (5.00,0.24);
  \draw[red, line width=0.8, draw opacity=\REDOPACITY] (1.00,0.24) to[out=90,in=90] (7.00,0.24);

  \draw[blue, line width=0.8] (1.00,-0.45) to[out=-80,in=-100] (5.00,-0.45);
  \draw[blue, line width=0.8] (3.00,-0.45) to[out=-80,in=-100] (7.00,-0.45);
  \draw[blue, line width=0.8] (5.00,-0.45) to[out=-80,in=-100] (7.00,-0.45);

    \draw[dashed, rounded corners=0.3cm] (3.8, -0.8) rectangle (6.2,0.8);

  \draw[dashed, rounded corners=0.55cm] (6.67,-4.3) rectangle (12.02,-1.7);
  
  \draw[dashed, rounded corners=0.1cm] (7.1,-3.6) rectangle (8.4,-2.7);

  \def\ixL{7.25}
  \def\ixR{11.35}
  \def\iy{-3.10}

  \draw[black, thick] (\ixL,\iy) -- (\ixR,\iy);
  \foreach \x in {\ixL,8.275,9.30,10.325,\ixR}{
    \draw[black, thin] (\x,\iy-0.14) -- (\x,\iy+0.14);
  }

  \foreach \a/\b in {7.52/8.03,8.55/9.06,9.58/10.09,10.61/11.12}{
    \draw[red][line width=1, draw opacity=0.6] (\a-0.05,\iy+0.20) -- (\b+0.05,\iy+0.20);
    \draw[blue][line width=1] (\a-0.05,\iy-0.35) -- (\b+0.05,\iy-0.35);
  }
\node[font=\footnotesize] at (7.52+0.255,\iy-0.18)  {$I_{20}$};
\node[font=\footnotesize] at (8.55+0.255,\iy-0.18) {$I_{21}$};
\node[font=\footnotesize] at (9.58+0.255,\iy-0.18) {$I_{22}$};
\node[font=\footnotesize] at (10.61+0.255,\iy-0.18) {$I_{23}$};
\node[font=\large] at (11.5, \iy+1) {$I_2$};
  \draw[red, line width=0.6, draw opacity=\REDOPACITY] (7.775,\iy+0.20) to[out=80,in=100] (8.805,\iy+0.20);
  \draw[red, line width=0.6, draw opacity=\REDOPACITY] (8.805,\iy+0.20) to[out=80,in=100] (9.835,\iy+0.20);
  \draw[red, line width=0.6, draw opacity=\REDOPACITY] (7.775,\iy+0.20) to[out=90,in=90] (10.865,\iy+0.20);

  \draw[blue, line width=0.6] (7.775,\iy-0.38) to[out=-80,in=-100] (9.835,\iy-0.38);
  \draw[blue, line width=0.6] (8.805,\iy-0.38) to[out=-90,in=-90] (10.865,\iy-0.38);
  \draw[blue, line width=0.6] (9.835,\iy-0.38) to[out=-90,in=-90] (10.865,\iy-0.38);

  \draw[-{To[length=3mm,width=2mm]}, thin, dashed] (6,-0.8) -- (6.9,-1.78);
  \draw[-{To[length=3mm,width=2mm]}, thin, dashed] (7.1,-3.3) -- (6,-3.8);
\node[font=\huge] at (5.4,-3.8) {$\dots$};

\end{tikzpicture}
      \caption{Proof of Lemma \ref{lem:SRT_lb}: construction of the colouring $f$, steps $1$ and $2$. In step $i$, we define the colouring for pairs of elements that belong to the same interval $I_\sigma$ of length $s^{t-i+1}$ but to distinct subintervals $I_{\sigma a}$, $I_{\sigma b}$ of length $s^{t-i}$.}
      \label{fig:lb}
    \end{figure}
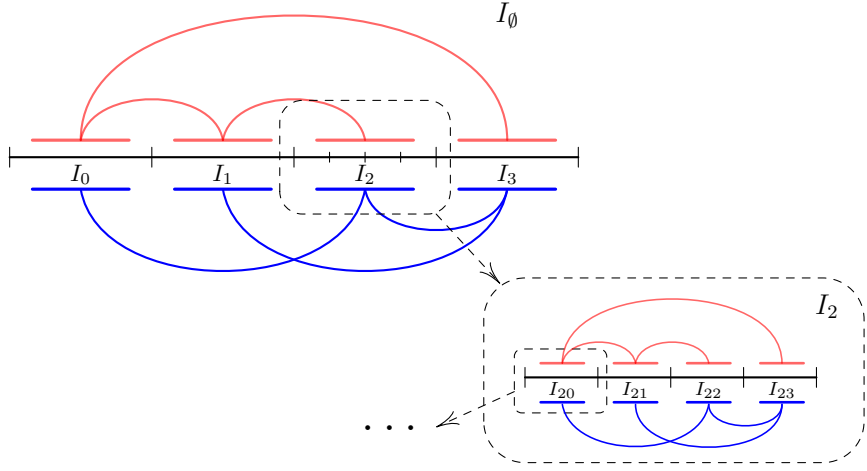
    
    In step $0$, we let $I_\emptyset$ be $[m] = \{0,\ldots,m-1\}$.

    Suppose that in step $i-1$, we have divided $[m]$ into $s^{i-1}$ 
    many intervals $I_{0\ldots0},\ldots, I_{(s-1)\ldots(s-1)}$, each of length $s^{t-i+1}$,
    and the colouring $f$ has already been defined for all pairs whose elements come from different intervals.
    We divide each interval $I_{\sigma}$ into $s$ many successive subintervals $I_{\sigma0}, \ldots, I_{\sigma(s-1)}$, each of length $s^{t-i}$.
    We colour pairs coming from two different subintervals $I_{\sigma a}, I_{\sigma b}$ of $I_{\sigma}$ using $g$: that is, if $x,y$ belong to $I_{\sigma a}$,
    $I_{\sigma b}$, respectively, where $a < b$, we
    set $f(x,y) := g(a,b)$.
    The colouring of pairs in which both elements come from the same interval $I_\tau$ with $|\tau| = i$ is left undefined at this stage.
    
    This completes the construction of $f$. Note that 
    the intervals $I_\sigma$ built in step $t$
    have length $s^0 = 1$, so after step $t$ no pair of distinct elements of $[m]$ is left without a colour.

    We now show that the colouring $f$ has property (a).
    Fix any $x \in [m]$, and for $i = 0,\ldots,t$ let $I_{\sigma_i}$ be the interval of length $s^{t-i}$ containing $x$ that appeared in step $i$ of the construction above.
    In this step, we define the value of $f(x,-)$ on positions corresponding to elements of 
    $I_{\sigma_{i-1}}\setminus I_{\sigma_i}$.
    Since we are actually colouring pairs based on which of $s$ distinct subintervals of $I_{\sigma_{i-1}}$ their elements belong to, the segment of $f(x,-)$ corresponding to $I_{\sigma_{i-1}}\setminus I_{\sigma_i}$ 
    consists of at most $s$ (in fact, at most $s-1$) blocks of successive 0's or successive 1's. So, altogether $f(x,-)$ consists of at most $st$ such blocks.

    It remains to show that $f$ has property (b).
    Let $H\subseteq [m]$ be homogeneous for $f$. We prove by induction on $i$ that, if $I_\sigma$ is an interval of length $s^i$ appearing in the construction, then $|H\cap I_\sigma|< (2\log s)^i$.
    As a direct consequence, $|H|=|H\cap I_\emptyset|<(2\log s)^t$.
    
    If $i=1$, then the colouring $f$ on $I_\sigma$ is exactly the same as $g$, so $|H\cap I_\sigma| < 2\log s$ by the choice of $g$.
    
    Now suppose $I_\sigma$ has length $s^{i+1}$.
    During the construction, we divided $I_\sigma$ into $s$ subintervals, and we coloured pairs consisting of elements from two distinct subintervals using $g$.
    So, fewer than $2\log s$ of the $s$ subintervals of $I_\sigma$ have non-empty intersection with $H$, because otherwise 
    $H \cap I_\sigma$ would not be homogeneous. By induction, each subinterval contains fewer than $(2\log s)^i$ many elements in $H$, so $|H\cap I_\sigma| < (2\log s)^{i+1}$. 
\end{proof}

\begin{corollary}\label{cor:srttt-qpoly-lb}
    $(\EA)$ For any  $k\in\IN$ larger than some fixed standard constant, there exists a colouring 
    $f \colon {[k^{(\log k)^{0.9}}]}^2\rightarrow 2$ that is at most $k$-unstable but has no homogeneous set of size $k$.
\end{corollary}

\begin{proof}
    Let $t=(\log k)^{0.95}$ and let $s$ be $k^{0.9}$ so that $\log s=0.9\log k$.
    Then, assuming $k$ is sufficiently large,
    \begin{align*}
        st & \le k^{0.9} \log k \le k,\\
        (2\log s)^t & \le 2^{(\log \log k+1)(\log k)^{0.95}}\leq  k,\\
        s^t & = 2^{0.9(\log k)^{1.95}}\geq 2^{(\log k)^{1.9}} = k^{(\log k)^{0.9}}.    
    \end{align*}

We obtain the required colouring $f$ by applying Lemma~\ref{lem:SRT_lb} to $s$ and $t$ as specified above.
\end{proof}

Corollary \ref{cor:srttt-qpoly-lb} tells us that 
there is a nontrivial, namely quasipolynomial, lower bound on Ramsey numbers even for the restricted situation in which we only consider somewhat stable colourings: more precisely, when looking for a 
homogeneous set of size $k$, we consider only colourings that are at most $k$-unstable. 
Our aim is now to show that there is a corresponding quasipolynomial upper bound as well. 

There is a well-known class of colourings for which  Ramsey numbers are at most quadratic. This class was given a definition in \cite{erdos-hajnal} in graph-theoretic terminology, which is probably more natural in this context, but we reformulate it in terms
of sets and colourings to maintain notational uniformity.

\begin{definition}[Erd\"{o}s--Hajnal~\cite{erdos-hajnal}]\label{def:very-simple}
A colouring $f \colon [S]^2 \to 2$ for a finite $S \subseteq \IN$ is \emph{very simple} if it belongs to the smallest class satisfying the following conditions:

\begin{enumerate}[(a)]
    \item if $S$ has one element, then the unique colouring of $[S]^2$ is very simple,
    \item if $f_1 \colon [S_1]^2 \to 2$ and 
    $f_2 \colon [S_2]^2 \to 2$ are both very simple, then each of the two colourings $g_0, g_1$ of $[S_1 \cup S_2]^2$ is very simple, where $g_b$ acts as $f_i$ on $[S_i]^2$ and has constant value $b$ on all pairs with one element in $S_1$ and one in $S_2$.
\end{enumerate}

\end{definition}

The following lemma is stated and proved 
(of course, in terms of graphs and without mention of provability in $\EA$) 
in \cite{erdos-hajnal}, where it is referred to
as well-known.

\begin{lemma}[{\cite[Lemma 1.4]{erdos-hajnal}}]\label{lem:vsgraph}
    $(\EA)$ If $f: [S]^2 \to 2$ is a very simple colouring, then there is a homogeneous set for $f$
    of size at least $\sqrt{|S|}$.
\end{lemma}
\begin{proof}
    It is proved by induction on $|S|$ that
    the product of the size of the largest homogeneous set with colour $0$ and the size of the largest homogeneous set with colour $1$ must be at least $|S|$. The inductive step is a straightforward case analysis that easily formalizes in $\EA$. 
\end{proof}

To prove our upper bound on Ramsey
numbers for stable colourings, we will show that
if a colouring is at most $k$-unstable but defined on a set of size considerably larger than $k$, then we can inductively construct a reasonably big subset on which the colouring is very simple. The following lemma provides the core of the inductive step in that construction.

\begin{lemma}\label{lem:grouping}
    $(\EA)$ Let $c\geq 2$, and let $k$ be larger than some fixed standard constant.
    Then for any $A\subseteq \IN$ with $|A|=k^c$ the following holds: if $f\colon [A]^2\imp 2$ is at most $k$-unstable, then there exist some $A_0, A_1\subseteq A$ with $\max A_0 < \min A_1$ such that $|A_0|=|A_1|= k^{c-2}$ and $f\restdto {A_0\times A_1}$ is constant.
    \end{lemma}

\begin{proof}
    We identify $A$ with $[k^c]$.
    Let $A_0' = [4k^{c-2}]$ and $A_1' = [k^c]\setminus A_0'$. 
    
    Assume that we sample a subinterval $I\subseteq A_1'$ of length $k^{c-2}$ uniformly at random. For each fixed $x\in A_0'$, let $E_{x,I}$ be the event ``$f(x,-)$ does not change value within $I$''.
    Then 
    \[1 - \Pr_{I}(E_{x,I}) \leq \frac{k(|I|-1)}{|A_1'|-|I|+1} \leq\frac{k^{c-1}}{k^c - 5k^{c-2}}\leq \frac{1}{2}.\]
    
    The first inequality above holds because if a given point at which $f(x,-)$ changes value lies within $I$, then the left end of $I$ must lie in one of at most $|I|-1$ specific positions among all the  $|A_1'|-|I|+1$ possible ones. The factor $k$ in the numerator corresponds to taking the union bound over at most $k$ points at which $f(x,-)$ changes. 
    
    So, $\Pr_{I}(E_{x,I})\geq 1/2$ for each fixed $x \in A_0'$. Thus, if we randomly choose \emph{both}
    $x \in A_0'$ and the interval $I$, we have 
    $\Pr_{x,I}(E_{x,I})\geq \frac{1}{2}$.
    Let $A_1$ be the subinterval $I$ of $A_1'$ that maximizes $\Pr_{x}(E_{x,I})$. By a routine averaging argument, we have
    \[\Pr_{x}(E_{x,A_1})\geq \frac{1}{2}.\]
    
    Let $A_0''$ be the subset of $A_0'$ such that $x\in A_0''$ iff $E_{x, A_1}$ holds.
    Then by the lower bound on $\Pr_{x}(E_{x,A_1})$ above, we have $|A_0''|\geq 2k^{c-2}$.
    
    Now, $f\restd_{\{x\}\times A_1}$ is a constant for each $x\in A_0''$. We choose $b \in \{0,1\}$ such that
    \[\{x\in A_0''\mid f\restd_{\{x\}\times A_1} \equiv b\}|\geq |\{x\in A_0''\mid f\restd_{\{x\}\times A_1} \equiv 1-b\}|\]
    and let $A_0 = \{x\in A_0''\mid f\restd_{\{x\}\times A_1} \equiv b\}$.
    
    Then $f\restd_{A_0\times A_1}$ has constant value $b$ and $|A_0| \ge k^{c-2}$ by the choice of $b$. Finally, we may leave out some elements of $A_0$ so as to ensure $|A_0|=|A_1|=k^{c-2}$.
\end{proof}

\begin{theorem}\label{thm:upperbound}
    $(\EA)$ Let $d \geq 1$, and let $k\in\IN$ be larger than some fixed standard constant. 
    If $f\colon [k^{2d}]^2\imp 2$ is at most $k$-unstable, then there is a homogeneous set $H$ for $k$ with $|H| \ge 2^{d/2}$.
\end{theorem}

\begin{proof}[Proof of Theorem~\ref{thm:upperbound}]
    By Lemma \ref{lem:vsgraph}, it is enough to show
    that there exists a set $A \subseteq [k^{2d}]$
    of size $2^{d}$ such that $f\restdto{[A]^2}$ 
    is very simple. To achieve this, we apply Lemma~\ref{lem:grouping} iteratively.
    For each string $\sigma\in 2^{\le d}$, we construct a set $A_\sigma\subseteq [k^{2d}]$ with $|A_\sigma| = k^{2d - 2|\sigma|}$ as follows:
    \begin{enumerate}[(i)]
        \item $A_\emptyset = [k^{2d}]$.
        \item If $A_\sigma$ is already defined for some $\sigma$ with $|\sigma| = d'<d$, then we obtain $A_{\sigma 0}$ and $A_{\sigma 1}$ by applying Lemma~\ref{lem:grouping} with $c := d'$ to
        $f\restdto{[A_\sigma]^2}$.
        Note that $f\restdto{[A_\sigma]^2}$ is 
        at most $k$-unstable as it is a restriction of $f$.
    \end{enumerate}
    
    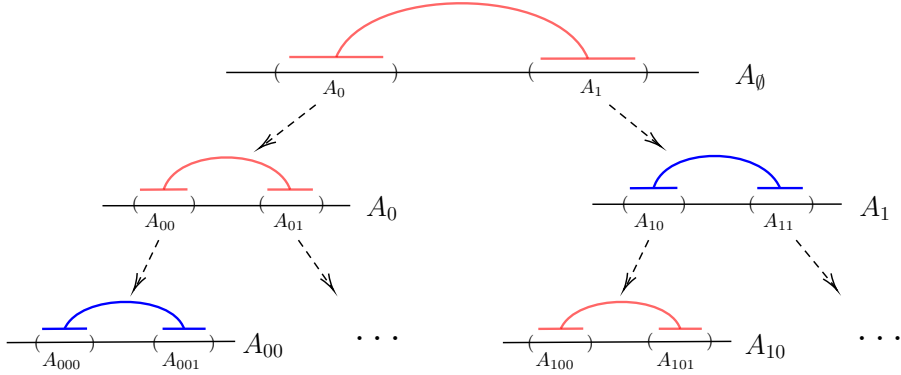
\begin{figure}
        \centering
        \tikzset{every picture/.style={line width=0.75pt}} 
\resizebox{1\textwidth}{!}{%
\begin{tikzpicture}[x=0.8pt,y=0.75pt,yscale=-1,xscale=1]

\def\REDOPACITY{0.6}

\draw    (160,70) -- (461,70) ;
\draw    (81.2,160.6) -- (239,160.5) ;
\draw [color=red  ,draw opacity=\REDOPACITY ][line width=1.2]   (230.08,59.5) .. controls (249.2,10.6) and (371.2,10.6) .. (390.35,60.56) ;
\draw [color=red  ,draw opacity=\REDOPACITY ][line width=1.2]   (120.2,149.6) .. controls (129,120.5) and (191,120.5) .. (200.7,149.6) ;
\draw  [dash pattern={on 4pt off 3pt}]  (217,92) -- (182.67,120.4) ;
\draw [shift={(181,121.5)}, rotate = 324.48] [color={rgb, 255:red, 0; green, 0; blue, 0 }  ][line width=0.75]    (10.93,-3.29) .. controls (6.95,-1.4) and (3.31,-0.3) .. (0,0) .. controls (3.31,0.3) and (6.95,1.4) .. (10.93,3.29)   ;
\draw  [dash pattern={on 4pt off 3pt}]  (404,92) -- (435.42,119.28) ;
\draw [shift={(437,120.5)}, rotate = 217.69] [color={rgb, 255:red, 0; green, 0; blue, 0 }  ][line width=0.75]    (10.93,-3.29) .. controls (6.95,-1.4) and (3.31,-0.3) .. (0,0) .. controls (3.31,0.3) and (6.95,1.4) .. (10.93,3.29)   ;
\draw  [dash pattern={on 4pt off 3pt}]  (206,184) -- (227.92,217.31) ;
\draw [shift={(228.5,218.31)}, rotate = 237.43] [color={rgb, 255:red, 0; green, 0; blue, 0 }  ][line width=0.75]    (10.93,-3.29) .. controls (6.95,-1.4) and (3.31,-0.3) .. (0,0) .. controls (3.31,0.3) and (6.95,1.4) .. (10.93,3.29)   ;
\draw [color=red  ,draw opacity=\REDOPACITY ][line width=1.2]    (200.16,59.5) -- (260,59.5) ;
\draw [color=red  ,draw opacity=\REDOPACITY ][line width=1.2]    (360.2,60.6) -- (403.2,60.54) -- (420.5,60.52) ;
\draw [color=red  ,draw opacity=\REDOPACITY ][line width=1.2]    (105.07,149.71) -- (135.33,149.49) ;
\draw [color=red  ,draw opacity=\REDOPACITY ][line width=1.2]    (186.2,149.6) -- (215.2,149.6) ;
\draw  [dash pattern={on 4pt off 3pt}]  (117,184) -- (101.93,217.23) ;
\draw [shift={(100.93,218.23)}, rotate = 297.82] [color={rgb, 255:red, 0; green, 0; blue, 0 }  ][line width=0.75]    (10.93,-3.29) .. controls (6.95,-1.4) and (3.31,-0.3) .. (0,0) .. controls (3.31,0.3) and (6.95,1.4) .. (10.93,3.29)   ;
\draw  [dash pattern={on 4pt off 3pt}]  (429,184) -- (413.87,217.2) ;
\draw [shift={(413,218)}, rotate = 298.94] [color={rgb, 255:red, 0; green, 0; blue, 0 }  ][line width=0.75]    (10.93,-3.29) .. controls (6.95,-1.4) and (3.31,-0.3) .. (0,0) .. controls (3.31,0.3) and (6.95,1.4) .. (10.93,3.29)   ;
\draw  [dash pattern={on 4pt off 3pt}]  (523,184) -- (548.84,217.37) ;
\draw [shift={(549,218)}, rotate = 234.61] [color={rgb, 255:red, 0; green, 0; blue, 0 }  ][line width=0.75]    (10.93,-3.29) .. controls (6.95,-1.4) and (3.31,-0.3) .. (0,0) .. controls (3.31,0.3) and (6.95,1.4) .. (10.93,3.29)   ;
\draw    (393.2,159.6) -- (552,159.5) ;
\draw [color=blue  ,draw opacity=1 ][line width=1.2]   (432.2,148.6) .. controls (440,119.5) and (501,119.5) .. (512.7,148.6) ;
\draw [color=blue  ,draw opacity=1 ][line width=1.2]    (417.07,148.71) -- (447.33,148.49) ;
\draw [color=blue  ,draw opacity=1 ][line width=1.2]    (498.2,148.6) -- (527.2,148.6) ;
\draw    (20.2,253.66) -- (165.71,252.82) ;
\draw [color=blue  ,draw opacity=1 ][line width=1.2]   (57.07,244.24) .. controls (67.28,219.92) and (124.95,219.92) .. (133.17,244.24) ;
\draw [color=blue  ,draw opacity=1 ][line width=1.2]    (42.76,244.34) -- (71.38,244.14) ;
\draw [color=blue  ,draw opacity=1 ][line width=1.2]    (119.47,244.24) -- (146.88,244.24) ;
\draw    (336.2,253.66) -- (481.71,252.82) ;
\draw [color=red  ,draw opacity=\REDOPACITY ][line width=1.2]   (373.07,244.24) .. controls (383.28,219.92) and (440.95,219.92) .. (449.17,244.24) ;
\draw [color=red  ,draw opacity=\REDOPACITY ][line width=1.2]    (358.76,244.34) -- (387.38,244.14) ;
\draw [color=red  ,draw opacity=\REDOPACITY ][line width=1.2]    (435.47,244.24) -- (462.88,244.24) ;

\draw (482.49,63.39) node [anchor=north west][inner sep=0.75pt]   [align=left][font=\Large] {$\displaystyle A_{\emptyset }$};
\draw (219.89,73.83) node [anchor=north west][inner sep=0.75pt]   [align=left][font=\normalsize] { $\displaystyle A_{0}$};
\draw (383.81,73.37) node [anchor=north west][inner sep=0.75pt]   [align=left] {$\displaystyle A_{1}$};
\draw (106.65,164.5) node [anchor=north west][inner sep=0.75pt]  [font=\normalsize] [align=left] {$\displaystyle A_{00}$};
\draw (188.49,164.5) node [anchor=north west][inner sep=0.75pt]  [font=\normalsize] [align=left] {$\displaystyle A_{01}$};
\draw (418.12,164.5) node [anchor=north west][inner sep=0.75pt]  [font=\normalsize] [align=left] {$\displaystyle A_{10}$};
\draw (500.63,164.5) node [anchor=north west][inner sep=0.75pt]  [font=\normalsize] [align=left] {$\displaystyle A_{11}$};
\draw (247.78,153.04) node [anchor=north west][inner sep=0.75pt]   [align=left][font=\Large] {$\displaystyle A_{0}$};
\draw (561.92,152.82) node [anchor=north west][inner sep=0.75pt]   [align=left][font=\Large] {$\displaystyle A_{1}$};
\draw (41.46,259) node [anchor=north west][inner sep=0.75pt]  [font=\normalsize] [align=left] {$\displaystyle A_{000}$};
\draw (117.62,259) node [anchor=north west][inner sep=0.75pt]  [font=\normalsize] [align=left] {$\displaystyle A_{001}$};
\draw (168.96,245.85) node [anchor=north west][inner sep=0.75pt]  [font=\Large] [align=left] {$\displaystyle A_{00}$};
\draw (188.49,62.23) node [anchor=north west][inner sep=0.75pt]   [align=left] {$\displaystyle ($};
\draw (263.95,62.23) node [anchor=north west][inner sep=0.75pt]   [align=left] {$\displaystyle )$};
\draw (353.98, 70) node  [xslant=-0.02] [align=left] {$\displaystyle ($};
\draw (423.5,62.22) node [anchor=north west][inner sep=0.75pt]   [align=left] {$\displaystyle )$};
\draw (98.41,152.5) node [anchor=north west][inner sep=0.75pt]  [font=\small] [align=left] {$\displaystyle ($};
\draw (135.46,152.5) node [anchor=north west][inner sep=0.75pt]  [font=\small] [align=left] {$\displaystyle )$};
\draw (178.92,152.43) node [anchor=north west][inner sep=0.75pt]  [font=\small] [align=left] {$\displaystyle ($};
\draw (217.2,152.5) node [anchor=north west][inner sep=0.75pt]  [font=\small] [align=left] {$\displaystyle )$};
\draw (410.41,152.5) node [anchor=north west][inner sep=0.75pt]  [font=\small] [align=left] {$\displaystyle ($};
\draw (447.46,152.5) node [anchor=north west][inner sep=0.75pt]  [font=\small] [align=left] {$\displaystyle )$};
\draw (490.92,152.43) node [anchor=north west][inner sep=0.75pt]  [font=\small] [align=left] {$\displaystyle ($};
\draw (529.2,152.6) node [anchor=north west][inner sep=0.75pt]  [font=\small] [align=left] {$\displaystyle )$};
\draw (36.17,246.5) node [anchor=north west][inner sep=0.75pt]  [font=\small] [align=left] {$\displaystyle ($};
\draw (71.2,246.5) node [anchor=north west][inner sep=0.75pt]  [font=\small] [align=left] {$\displaystyle )$};
\draw (111.28,246.5) node [anchor=north west][inner sep=0.75pt]  [font=\small] [align=left] {$\displaystyle ($};
\draw (147.47,246.5) node [anchor=north west][inner sep=0.75pt]  [font=\small] [align=left] {$\displaystyle )$};
\draw (355.46,259) node [anchor=north west][inner sep=0.75pt]  [font=\normalsize] [align=left] {$\displaystyle A_{100}$};
\draw (432.62,259) node [anchor=north west][inner sep=0.75pt]  [font=\normalsize] [align=left] {$\displaystyle A_{101}$};
\draw (488.96,246.5) node [anchor=north west][inner sep=0.75pt]  [font=\Large] [align=left] {$\displaystyle A_{10}$};
\draw (350.17,246.6) node [anchor=north west][inner sep=0.75pt]  [font=\small] [align=left] {$\displaystyle ($};
\draw (388.2,246.56) node [anchor=north west][inner sep=0.75pt]  [font=\small] [align=left] {$\displaystyle )$};
\draw (426.28,245.61) node [anchor=north west][inner sep=0.75pt]  [font=\small] [align=left] {$\displaystyle ($};
\draw (464.47,245.76) node [anchor=north west][inner sep=0.75pt]  [font=\small] [align=left] {$\displaystyle )$};
\draw (240,248) node [anchor=north west][inner sep=0.75pt]  [font=\huge] [align=left] {$\displaystyle \dotsc $};
\draw (560,248) node [anchor=north west][inner sep=0.75pt]  [font=\huge] [align=left] {$\displaystyle \dotsc $};
\end{tikzpicture}
}
        \caption{Proof of Theorem \ref{thm:upperbound}: the construction of $A_\sigma$ for $|\sigma|\leq 2$. Note that in general, $A_{\sigma0}$ is not a convex subset of $A_\sigma$.}
        \label{pic:2}
    \end{figure}
    
The construction is illustrated in Figure~\ref{pic:2}.    
We then have the following for any $\sigma\in 2^{\le d}$:

    \begin{itemize}
        \item $|A_\sigma|=k^{2d-2|\sigma|}$,
        \item if $|\sigma| < d$, then $A_\sigma\supseteq A_{\sigma  i}$ for $i\in\{0,1\}$,
        and $\max A_{\sigma 0} < \min A_{\sigma 1}$,
        \item if $|\sigma| < d$, then $f\restd_{A_{\sigma 0}\times A_{\sigma 1}}$ is constant.
    \end{itemize}
    
        We can prove by reverse induction on $|\sigma|$ 
    that $f$ restricted to pairs from the set
    \[\bigcup_{\tau \in 2^d \colon \sigma \sqsubseteq \tau} A_\tau\]
    is very simple. The induction is also very simple: the base step, for $|\sigma| = d$, is item (a) from the definition of very simple colourings, and the induction step is item (b).
    
    For $\sigma = \emptyset$, we conclude that $f$ restricted to $\left[\bigcup_{\tau \in 2^d} A_\tau\right]^2$ is very simple. 
    But $|\bigcup_{\tau \in 2^d} A_\tau| = 2^d$.
\end{proof}

\begin{corollary}\label{cor:srttt-qpoly-ub}
    $(\EA)$ For any $k \in \IN$ larger than some fixed standard constant, 
    every $f\colon[k^{4\log k}]^2\imp 2$ that is
    at most $k$-unstable 
    has a homogeneous subset of cardinality $k$. 
\end{corollary}
\begin{proof}
    Apply Theorem~\ref{thm:upperbound} with $d := 2\log k$.
\end{proof}

\begin{remark}\label{rem:ads}
    Ludovic Patey (private communication) has pointed out that in the proof of Theorem~\ref{thm:upperbound}, the set $A$ of size $2^d$ on which the colouring $f$ is very simple has the following property: one may require that in each step of the inductive construction from Definition \ref{def:very-simple} witnessing that $f\restdto{[A]^2}$ is very simple, the sets $S_1, S_2$ satisfy $\max S_1 < \min S_2$. As a consequence, $f\restdto{[A]^2}$ is not merely very simple, but also transitive in the sense that for $x<y<z$ in $A$, the equality $f(x,y) = f(y,z)$ implies $f(x,z) = f(x,y)$. In other words, $f\restdto{[A]^2}$ is an instance of $\ads$. 
    
    Thus, instead of invoking Lemma \ref{lem:vsgraph}, one could complete the proof of Theorem~\ref{thm:upperbound} by appealing to the well-known quadratic upper bounds on Ramsey numbers associated with $\ads$. Also, for large enough $k$, any $f\colon[k^{2\log k}]^2\imp 2$ that is at most $k$-unstable is simultaneously transitive and very simple on some set of size $k$. 
\end{remark}

In the next section, we will study what the
upper bound expressed in Corollary \ref{cor:srttt-qpoly-ub}
means for the relationship between proof size in 
$\RCA^*_0 + \srttt$ and in $\RCA^*_0$.
Before turning to that topic, we point out 
what the lower bound of Corollary \ref{cor:srttt-qpoly-lb}
implies concerning the closure properties of $\I^0_1$.

\begin{theorem} \label{thm:srttt->i01-qpclosed} Over $\rcas$, the principle $\srttt$
implies quasipolynomial closure of $\I^0_1$.
That is, $\rcas + \srttt$ proves that $\I^0_1$ is closed under $\omega_1$.
\end{theorem}

\begin{proof}
The argument 
is similar to the one proving Theorem \ref{thm:crttt->i01-expclosed}.
Let $(L, \cal X)$ be a model of $\rcas$ with $\I^0_1(L, \cal X)$ not closed under $\omega_1$. We take a proper exponential $\Sigma^0_1$-cut $M$ in $(L, \cal X)$,
and we need to prove that $(M, \cod(L/M)) \not \vDash \srttt$. As in the proof of Theorem \ref{thm:crttt->i01-expclosed}, we write $\I^0_1$ for the common $\I^0_1$ of the two structures $(L, \cal X)$ and $(M, \cod(L/M))$, and we 
write $\widehat S$ for $S \cap M$ whenever $S \in \cal X$.

The lack of closure under $\omega_1$ implies that
there is some $k \in M$ such that $k \in \I^0_1 < k^{{(\log k)}^{0.9}} =: \ell$. We fix an $L$-finite set $A = \{a_0, \ldots, a_{\ell-1}\}$ with $\widehat A$ cofinal in $M$, and by Corollary \ref{cor:srttt-qpoly-lb} 
we take $f \colon [A]^2 \to 2$ that is at most $k$-unstable but has no homogeneous set of size $k$.

Then, since $k \in I^0_1$, the colouring $\widehat f: [\widehat A]^2 \to 2$ has to be stable: 
for every $a_i \in \widehat A$, the sequence $f(a_i,-)$
contains at most $k$ positions where it switches between 0 and 1, so the set of those positions cannot be cofinal in $M$. But similarly, given that every homogeneous set for $f$ has cardinality less than $k$, no homogeneous set for $\widehat f$ can be cofinal in $M$. Thus,  
$(M, \cod(L/M)) \not \vDash \srttt$.
\end{proof}

\begin{remark}
Using the upper bound of Corollary \ref{cor:srttt-qpoly-ub}
and a routine argument like the one in the proof of
\cite[Theorem 3.16]{fkk:weakcousins}, one can show that
$\srttt$ does not imply any superquasipolynomial closure
properties of $\I^0_1$, in the following sense.
If $g \colon \IN \to \IN$ is a computable function such that for every $\ell \in \omega$ there is $n \in \omega$ satisfying $\EA \vdash \forall x \! \ge \! n \, (g(x) \ge x^{(\log x)^\ell})$, then $\rcas + \srttt$ does not prove that $\I^0_1$ is closed under $g$.
This is also implicitly shown by the forcing construction
discussed in Section \ref{sec:forcing}.
\end{remark}

\begin{remark}
Theorem \ref{thm:srttt->i01-qpclosed} implies that $\rcas + \srttt$ is not $\Pi_4$-conservative over $\rcas$, by an argument analogous to that proving Corollary \ref{cor:crttt-not-conservative} but with $\exp$ replaced
by $\omega_1$. 

However, the lack of $\Pi_4$-conservativity of $\srttt$
over $\rcas$ was already implicitly proved in \cite{kky:infinity-weak}.
In Corollary 4.4 of that paper, it is shown that
$\rcas + \rttt$ proves the $\Pi_4$ sentence 
$\mathrm{C}\Sigma_2$, 
which does not follow from $\rcas$ alone. The only property
of $\rttt$ needed for that argument is that
$\rca + \rttt$ implies $\bd \Sigma_2$; but $\srttt$ has that property as well.
\end{remark}

\section{Polynomial simulation}\label{sec:forcing}

The aim of this section is to prove the following result.

\begin{theorem}\label{thm:non-speedup}
The theory $\wkls + \srttt$ is polynomially simulated
by $\rcas$ w.r.t.~$\forall\Pi^0_3$ sentences.
\end{theorem}

The proof of Theorem \ref{thm:non-speedup} -- or, at least,
the proof we present here -- 
relies on Corollary \ref{cor:srttt-qpoly-ub}
and the technique of forcing interpretations, introduced originally (\emph{avant la lettre}) in \cite{avigad:formalizing-forcing}. 

The high-level structure of our argument is modelled closely on the proof of an analogous polynomial simulation result
for $\cac$ given by Kowalik in \cite{kowalik:cac}. 
However, there are two significant modifications, due to the fact that
the first-order part of $\srttt$ is more complex than that
of $\cac$ and to the quasipolynomial rather than polynomial nature of our bounds in Section \ref{sec:bounds}. 
Since the routine details of forcing interpretations are often highly tedious to describe even when the idea is quite clear, we present the proof in the form of a somewhat lengthy sketch focusing on the novel aspects of our argument.
More precisely, we first review the concept of forcing interpretations, and then we give a rather careful summary of the polynomial simulation proof, providing full details only in those places where our reasoning diverges from 
the argument of \cite{kowalik:cac}. 

\paragraph{Forcing interpretations.} In general, the idea of a forcing interpretation of a theory $T$ in a theory $S$ is that one defines in $S$ a notion of `$\psi$ is forced' for $\mathcal{L}_T$-formulas $\psi$, in such a way that $S$ can prove that each axiom of $T$ is forced. Then a $T$-proof of a sentence $\theta$ can be transformed 
into an $S$-proof of `$\theta$ is forced'. 
If, in addition, $\theta$ is actually in $\mathcal{L}_S \cap \mathcal{L}_T$ and we can construct an $S$-proof
of (`$\theta$ is forced' $\rightarrow \theta$),
then we obtain an $S$-proof of $\theta$ by modus ponens. Assuming that proofs verifying various properties of the forcing interpretation can be constructed in polynomial time, the transformation of the $T$-proof of $\theta$ into the $S$-proof is polytime as well.

In the form proposed in \cite[Section 1]{kwy:ramsey-proof-size} and then used in \cite{kowalik:cac}, to provide a forcing interpretation $\tau$ of $T$ in $S$ one has to define in $S$ the following: a set $\Cond_\tau$ of forcing conditions, a relation $\condle_\tau$ on the conditions, a set $\Name_\tau$ of names (intended to name elements of the generic $\mathcal{L}_T$-structure being described by the forcing), 
and the forcing relations 
$p \Vdash_\tau v \defd$ and $p \Vdash_\tau \alpha(\overline v)$, where $p$ is a condition, $v,\overline v$ are names, and $\alpha$ is an \emph{atomic} $\mathcal{L}_T$-formula (the intuitive meaning of $p \Vdash_\tau v \defd$ is that $v$ will in fact name an element of the generic structure as long as $p$ is in the generic filter). These definitions have to satisfy a number of well-behavedness requirements 
provably in $S$: 
roughly, $\condle_\tau$ should be a preorder, 
the forcing relation(s) should be monotone w.r.t.~$\trianglelefteqslant_\tau$, the equality axioms should be forced, and $p$ should force a statement (be it $v \defd$ or $\alpha(\overline v)$) if it is densely forced below $p$.
(For a precise formulation of the requirements, see
\cite[Definition 3.2]{kowalik:cac}; or \cite[Definition 1.8]{kwy:ramsey-proof-size}, where, however, the treatment of atomic formulas has to be more careful due to greater intended generality).

The definition of $p \Vdash_\tau \theta$ is then extended to non-atomic formulas $\theta$ using standard inductive conditions:
\begin{enumerate}[(i)]
    \item $p \Vdash_\tau \neg \theta$ iff $q \not \Vdash_\tau \theta$ for all $q \condle_\tau p$;
    \item $p \Vdash_\tau  (\theta \to \eta)$ iff for every $q \condle_\tau p$ such that $q \Vdash_\tau \theta$ there is $r \condle_\tau q$ such that $r \Vdash_\tau \eta$;
    \item $p \Vdash_\tau \forall x\, \theta(x)$ iff for every name $w$ and every $q \condle_\tau p$ such that $q \Vdash_\tau  w \defd$, there is $r \condle_\tau q$ such that $r \Vdash_\tau  \theta(w)$.
\end{enumerate}
(For simplicity, we suppress names corresponding to potential free variables in $\theta, \eta$ other than $x$ in (iii)). 
The final requirement that has to be satisfied 
for $\tau$ to be a forcing interpretation of $T$
is that $S$ proves $\Vdash_\tau \psi$ for each axiom $\psi$ of $T$.
Here $\Vdash_\tau \ldots$ stands for $\forall p \! \in \! \Cond_\tau \, (p \Vdash_\tau \ldots)$.

\begin{remark}
As discussed in the Preliminaries,
our official language does not contain connectives other than $\neg, \to, \forall$, but typical inductive conditions also apply to forcing formulas involving other connectives if they
are defined from $\neg, \to, \forall$ in the usual way. In particular, $p$ forces an existential statement iff densely below $p$ a specific instance is forced.    
\end{remark}

If the existence of $\tau$ is to imply that
$T$ is polynomially simulated by $S$ w.r.t.~a class of sentences $\Gamma$, two additional requirements have to be met. Firstly, $\tau$~has to be a \emph{polynomial forcing interpretation of $T$}, which in the simple case of finitely axiomatized $T$ means that the $S$-proofs of some well-behavedness properties of $\Vdash_\tau$ on atoms
(for example, that the equality axioms are forced) can be constructed in polytime. Secondly, $\tau$ has to be \emph{polynomially $\Gamma$-reflecting}: there has to be a polytime procedure which, given a sentence $\gamma \in \Gamma$, outputs an $S$-proof of $(\Vdash_\tau \gamma) \rightarrow \gamma$.
(The fact that the existence of a 
{polynomially $\Gamma$-reflecting} forcing interpretation
of $T$ in $S$ implies polynomial simulation of $T$ by $S$ w.r.t.~$\Gamma$ is stated as \cite[Theorem 3.8]{kowalik:cac} 
or \cite[Theorem 1.18]{kwy:ramsey-proof-size}, though the idea goes back to \cite[Section 10]{avigad:formalizing-forcing}.)

\paragraph{Structure of the polynomial simulation argument.} 
We now discuss the structure of the proof of Theorem \ref{thm:non-speedup}.
The result of \cite{kwy:ramsey-proof-size} that $\wkl + \rttt$ is polynomially simulated by $\rca$ w.r.t.~$\forall \Pi^0_3$ sentences means that for any consequence $\xi$ of $\rttt$, polynomial simulation of $\wkls + \xi$ by $\rcas$ 
follows from polynomial simulation of $\wkls + \xi$ 
by $\rcas + \neg \ind \Sigma^0_1$ 
(all still at the $\forall \Pi^0_3$ level).
To achieve this for $\xi$ equal to $\srttt$, 
we introduce an intermediate auxiliary theory
${\EA} + {\SC_{\omega_1}}$, where $\SC_{\omega_1}$ is the following axiom in the language $\mathcal{L}_\PA \cup \{\II\}$ for a unary predicate $\II$:
\begin{quote}
$\II$ is a proper cut closed under $\omega_1$ \\ 
such that for any $x$ there is some $v>\II$ for which $2_v(x)$ exists.
\end{quote}
This is a stronger version of the axiom $\SC$ considered in 
\cite{kowalik:cac}, which required only closure of $\II$ 
under multiplication rather than under $\omega_1$. The reason why this strengthening turns out to be relevant is that the bound from Corollary \ref{cor:srttt-qpoly-ub} is quasipolynomial, in contrast to the polynomial bound on Ramsey numbers for $\cac$ that follows from Dilworth's Theorem.

We need polynomial forcing interpretations of ${\EA} + {\SC_{\omega_1}}$ 
in $\rcas + \neg \ind \Sigma^0_1$ and of $\wkls + \srttt$ in $\EA + {\SC_{\omega_1}}$. We begin with the former.

\paragraph{Interpreting the auxiliary theory.}
In order to (forcing-)interpret $\EA + {\SC}$ (without the additional requirement that $\II$ be closed under $\omega_1$)
in $\rcas + \neg \ind \Sigma^0_1$, we follow \cite{kowalik:cac} by making a case distinction. 
Let $\Sigma^0_1$-$\LPC$ be the statement ``there exists a least proper $\Sigma_1^0$-definable cut''.
If $\neg \Sigma^0_1$-$\LPC$ holds, then there is 
a straightforward (non-forcing) interpretation in which all the arithmetical operations are interpreted identically and $\II$ is interpreted as $\I^0_1$. On the other hand,
under $\Sigma^0_1$-$\LPC$ we can use a forcing interpretation $\tau_1$ that describes a \emph{restricted ultrapower} of the first-order universe (modulo a generic ultrafilter)
with $\II$ interpreted as the cut generated by $\I^0_1$ of the ground model. Conditions are unbounded subsets of $\IN$ (potential elements of the ultrafilter),
the order $\condle_{\tau_1}$ is inclusion, names are total 
functions from $\IN$ to $\IN$. 
The details of the forcing
relation are exactly as in \cite{kowalik:cac}
and will not be referred to below. 
Thus, we do not present them here except for noting that $p \Vdash_{\tau_1} t(\overline v) \in \II$ if there is $k \in \I^0_1$ such that $t(\overline v(x)) \le k$ holds for all but finitely many $x \in p$ (intuitively, the element of the ultrapower given by $t(\overline v)$, where $\overline v$ is a tuple of functions, is bounded by an element of $\I^0_1$, identified as usual with a constant function).

Making the additional step from $\SC$ to ${\SC_{\omega_1}}$ is actually quite easy, thanks to Solovay's technique of shortening cuts. 
There is a definable subcut of $\I^0_1$ that is provably closed under $\omega_1$, namely the cut defined by \[\varphi(x) := \forall y\, (y \in \I^0_1 \to 2^{\log x \cdot \log y} \in \I^0_1).\] 
Both in the $\neg\Sigma^0_1$-$\LPC$ and in
the $\Sigma^0_1$-$\LPC$ case, substituting $\varphi$ for
the definition of $\I^0_1$ throughout gives rise to a polynomial forcing interpretation of ${\EA} + {\SC_{\omega_1}}$ (which does not actually involve forcing
in the $\neg\Sigma^0_1$-$\LPC$ case).
In the $\Sigma^0_1$-$\LPC$ case, as a gesture to  subscript enthusiasts, we will refer to the resulting forcing interpretation
as $\tau_{1, \omega_1}$.

Summarizing this part of the argument, we get: 

\begin{lemma}[cf.~\cite{kowalik:cac}, Lemma 4.4]\label{lem:forcing-int-1}
The forcing interpretation 
$\tau_{1, \omega_1}$ 
is a polynomial forcing interpretation of ${\EA} + {\SC_{\omega_1}}$ in $\rcas + \neg \ind \Sigma^0_1 + \Sigma^0_1\text{-}\LPC$.
There is also a (non-forcing) interpretation of ${\EA} + {\SC_{\omega_1}}$ in $\rcas + \neg \ind \Sigma^0_1 + \neg \Sigma^0_1\text{-}\LPC$ that is the identity on $\mathcal{L}_\PA$.
\end{lemma}

\begin{remark}
Although the step from $\SC$ to $\SC_{\omega_1}$
described above is very simple from a technical point of view,
the possibility of shortening cuts is the main reason why a \emph{quasipolynomial} (or even quasiquasipolynomial, etc.) bound on the finite version of a Ramsey-theoretic 
principle makes it possible to prove a \emph{polynomial} simulation over $\rcas$.
\end{remark}

\paragraph{Interpreting $\srttt$.}
The formulas defining our forcing interpretation $\tau_2$ of $\wkl + \srttt$ in $\EA + {\SC_{\omega_1}}$ are exactly
the same as those used in \cite{kowalik:cac} to 
forcing-interpret $\wkl + \cac$ in ${\EA} + {\SC}$.
The difference is that we now have a stronger
theory of the ground model,
guaranteeing better closure properties of $\II$, which will
allow us to interpret the additional principle $\srttt$.

The forcing interpretation $\tau_2$
describes a \emph{generic cut}, 
with the coded sets as the second-order structure, 
satisfying $\wkls$. The forcing conditions of $\tau_2$
are finite sets $p$ such that 
\begin{enumerate}[(i)]
    \item $|p|>\II$.
    \item $\fain {x,y} p (x<y\rightarrow 2^x<y)$.
\end{enumerate}
A condition $q$ extends $p$ in the sense of $\condle_{\tau_2}$ if $q\subseteq p$. Intuitively,
a condition $p$ approximates a cut between $\min p$ and $\max p$ such that elements of $p$ appear arbitrarily high in the cut.
Requirement (i) on $p$ guarantees that $\I_1^0$ of the cut contains $\II$ (in fact, it will be exactly $\II$ by genericity); requirement (ii) guarantees that $\exp$ holds in the cut.

Since $\wkls$ is a two-sorted theory, we need two kinds of names: first-order names $v$ are natural numbers, second-order names $V$ are also natural numbers but viewed as the finite sets they code. It is always the case that $p \Vdash_{\tau_2} V\defd$, while $p \Vdash_{\tau_2} v\defd$ if $p\cap[0,v]$ is not a condition: the idea is that $p\cap[0,v]$ is then small, so the cut approximated by $p$ will be above $v$. Finally, for an atom $\alpha$ we have $p \Vdash_{\tau_2} \alpha(\overline v,\overline V)$ 
if $p \Vdash_{\tau_2} \overline v\defd$ 
and $\alpha(\overline v,\overline V)$ holds (where an atom of the form $t \in V$ is understood according to the usual Ackermann interpretation of set theory in arithmetic).

\begin{lemma}[cf.~{\cite[Lemma 4.9]{kowalik:cac}}]\label{lem:forcing-int-2}
The forcing interpretation $\tau_2$ is a polynomial
forcing interpretation of $\wkls + \srttt$ in ${\EA} + {\SC_{\omega_1}}$. 
\end{lemma}

\begin{proof}
The verification that $\tau_2$ is a forcing interpretation of $\wkls$ is just like in \cite[Lemma 4.9]{kowalik:cac}. (In broad strokes: requirement (ii) from the definition
of $\Cond_{\tau_2}$ ensures that the generic cut is exponentially closed, and every exponentially closed proper cut in a model of $\ind \Delta_0$ is a model of $\wkls$.)

Also the fact that $\tau_2$ is a polynomial forcing interpretation is proved like in \cite{kowalik:cac}, 
because the only new thing we have to do is to provide 
a (necessarily constant-sized) proof in $\EA + \SC_{\omega_1}$ that the single statement $\srttt$ is forced.

We reason inside $\EA+{\SC_{\omega_1}}$, and we drop the subscript ${}_{\tau_2}$, writing $\Vdash$, $\condle$ etc.
Our arguments below will make use of the following intuitively obvious fact (see \cite{kwy:ramsey-proof-size} or \cite{kowalik:cac} for a proof): if $p \in \Cond$, then
$p \Vdash \text{`}p \text{ is unbounded'}$, where $\text{`}p \text{ is unbounded'}$ stands for $\forall x\, \exists y\, (y \ge x \land y \in p)$.

Since all logical equivalences are forced, we can 
assume that $\srttt$ states:
\begin{equation}\label{eqn:srt}
\forall f \,(\text{`}f \text{ stable'} \to 
\exists H \, (\text{`}H \text{ unbounded'}
 \land \text{`}H \text{ homogeneous for } f\text{'})).
\end{equation}

Assume that $p$ is a condition, $f$ is a second-order name 
and $p$ forces that $f$ is a function from $[\IN]^2$ to $2$;
this lets us also assume w.l.o.g.~(by possibly shrinking $p$, thus passing to a stronger condition) that $f\restdto{[p]^2}$
is in fact a 0-1 valued function. We want to show that
$p$ forces the statement within parentheses in (\ref{eqn:srt}).

We make the following claim.

\begin{claim}\label{claim:stable_forcing}
    If $p \Vdash \text{`}f \text{ stable'}$, then 
    there exists some $k\in\II$ such that $f\restdto{[p]^2}$ is at most $k$-unstable.
\end{claim}

\begin{proofofclaim}
    Suppose otherwise.  Then by $\Delta_0$ overspill, there
    is some $x_0 \in p$ and $\ell > \II$ such that 
    $f\restdto{[p]^2}(x_0,-)$ changes value $\ell$ times.
    Let $q$ be the subset of $p$ consisting of $x_0$
    and all the places where $f\restdto{[p]^2}(x_0,-)$
    changes value. Then $|q| = \ell +1$, so $q$ is 
    a condition and $q \condle p$.
    We will show that $q \Vdash \text{`}f \text{ not stable'}$, which will contradict our assumption that $p\Vdash \text{`}f \text{ stable'}$.

    Note that $x_0 = \min q$, so $q \Vdash x_0 \defd$. 
    Since $q \Vdash \text{`}q \text{ unbounded'}$,
    to prove that $q$ forces that $f$ is not stable it suffices to show:
    \begin{equation}\label{eqn:unstable}
    q \Vdash \forall x \, (x \in q \land x > x_0 \to 
    \exists y \, (y > x \land f(x_0,x) \neq f(x_0,y)).    
    \end{equation}
    Let $q = \{x_0 < x_1 < \ldots < x_\ell\}$. Assume that for some $r \condle q$ and some name $w$, we have 
    $r \Vdash (w \in q \land w > x_0)$. Then by the definition of forcing for atoms, we must have $w = x_i$ for some $0 < i < \ell$, and $r \Vdash x_i \defd$.
    But this means that $r \Vdash x_{i+1} \defd$ as well,
    and then $r \Vdash (x_{i+1} > x_i \land f(x_0,x_i) \neq f(x_0,x_{i+1}))$.
    This completes the proof of (\ref{eqn:unstable})
    and hence of the Claim.
\end{proofofclaim}

To show that our condition $p$ forces (\ref{eqn:srt}), 
assume that $p \Vdash \text{`}f \text{ stable'}$.
By the axiom $\SC_{\omega_1}$, 
the cut $\II$ is closed under $\omega_1$, 
so by possibly removing some elements from $p$, 
we may also assume that $|p|=\ell^{4\log \ell}$ for some $\ell>\II$.
By Claim~\ref{claim:stable_forcing}, the colouring 
$f\restdto{[p]^2}$ is at most $\ell$-unstable.
Thus, we can apply Corollary \ref{cor:srttt-qpoly-ub}
to conclude that there is some $q \subseteq p$ of 
cardinality $\ell$ homogeneous for $f$.

Since $\ell > \II$, we know that $q \in \Cond$, which means that $q \condle p$. We have 
$q \Vdash \text{`}q \text{ unbounded'}$,
and it is easy to verify that 
$q \Vdash \text{`}q \text{ homogeneous for } f\text{'}$,
given that $q$ is, in fact, homogeneous for $f$.
This suffices to prove that $p$ forces (\ref{eqn:srt}).
\end{proof}

\paragraph{Reflection.}
This part of the argument is just like in \cite{kowalik:cac}
except for the additional appearance of $\omega_1$. 
We do not prove a polynomial reflection property for each forcing interpretation separately (note that this would involve changing the class of sentences involved, 
as $\rcas$ and $\wkls + \srttt$ are two-sorted theories, 
while the intermediate theory ${\EA} + {\SC_{\omega_1}}$ 
is one-sorted). 
Instead, reflection is proved directly for the composition:

\begin{lemma}[cf.~\cite{kowalik:cac}, Lemma 4.11]\label{lem:reflection}
There exists a polytime procedure which, given an
$\exists \Sigma^0_3$ sentence $\xi$, outputs
a proof in 
${\rcas} + {\neg \ind \Sigma^0_1} + {\Sigma^0_1\text{-}\LPC}$ of
\[\xi \to (\Vdash_{\tau_{1,\omega_1}}(\not \Vdash_{\tau_2} \neg \xi)).\]
\end{lemma}

\begin{remark}\label{rem:reflection+notlpc}
An analogous result holds with $\Sigma^0_1\text{-}\LPC$
replaced by $\neg\Sigma^0_1\text{-}\LPC$ and $\tau_{1,\omega_1}$
replaced by the interpretation of ${\EA} + {\SC_{\omega_1}}$ 
in $\rcas + \neg \ind \Sigma^0_1 + \neg \Sigma^0_1\text{-}\LPC$ discussed in Lemma \ref{lem:forcing-int-1}.    
\end{remark}

\paragraph{Completing the argument.} 
We can now review how the pieces are put together to complete the proof of the polynomial simulation.

\begin{proof}[Proof of Theorem \ref{thm:non-speedup}]
By \cite[Theorem 2.1]{kwy:ramsey-proof-size} and an
obvious argument by cases, it 
is enough to prove that $\wkls + \srttt$ is polynomially simulated by $\rcas + \neg \ind \Sigma^0_1$ w.r.t.~$\forall \Pi^0_3$. By another case distinction, it suffices
to prove polynomial simulation by $\rcas + \neg \ind \Sigma^0_1 + \Sigma^0_1\text{-}\LPC$ and separately by 
$\rcas + \neg \ind \Sigma^0_1 +\neg\Sigma^0_1\text{-}\LPC$.

Let us consider the case of $\Sigma^0_1\text{-}\LPC$.
Assume that we are given a proof in $\wkls + \srttt$
of a $\forall \Pi^0_3$ sentence $\psi$. 
By Lemmas \ref{lem:forcing-int-1}, \ref{lem:forcing-int-2}
and the properties of forcing interpretations,
we have a polynomial-time procedure that outputs a proof in $\rcas + \neg \ind \Sigma^0_1 + \Sigma^0_1\text{-}\LPC$ of  
\begin{equation*}\label{eqn:forces-psi}
\Vdash_{\tau_{1,\omega_1}} (\Vdash_{\tau_2} \psi).    
\end{equation*}
On the other hand, by Lemma \ref{lem:reflection}, another polytime procedure produces a proof in $\rcas + \neg \ind \Sigma^0_1 + \Sigma^0_1\text{-}\LPC$
of 
\begin{equation*}\label{eqn:forces-not-psi}
\neg \psi \to (\Vdash_{\tau_{1,\omega_1}} (\not \Vdash_{\tau_2} \psi)).    
\end{equation*} 
Putting these two proofs together,
we obtain a proof in $\rcas + \neg \ind \Sigma^0_1 + \Sigma^0_1\text{-}\LPC$ of 
$\neg \psi \to (\Vdash_{\tau_{1,\omega_1}} \bot)$,
which is easily seen to imply $\psi$.

The argument for $\neg \Sigma^0_1\text{-}\LPC$ is similar,
but Lemma \ref{lem:reflection} is replaced by Remark~\ref{rem:reflection+notlpc}.
\end{proof}

\section*{Acknowledgement}

We are grateful to Keita Yokoyama for helpful comments and questions, in particular the ones that inspired Remark \ref{rem:fin-coh}, and to Ludovic Patey for the comment that led to Remark \ref{rem:ads}.

This work was supported by grant no.~2023/49/B/ST1/02627 of the National Science Centre, Poland.

\bibliographystyle{plain}
\bibliography{crt-srt}

\end{document}